\documentclass[11pt,reqno,a4paper]{amsart}

\usepackage{etoolbox}
\usepackage{enumerate}
\usepackage{amssymb}
\usepackage[initials]{amsrefs}
\usepackage{amsfonts}
\usepackage{comment}
\usepackage{courier} 
\usepackage[utf8]{inputenc}      
\usepackage[T1]{fontenc}                                                    
\usepackage{hyperref}
\usepackage{slashed}
\usepackage{tikz-cd}

\numberwithin{equation}{section}

\newcommand{\cD}{\mathcal{D}}

\newcommand{\cS}{\mathcal{S}}

\newcommand{\bN}{\mathbb{N}}

\newcommand{\bR}{\mathbb{R}}

\newcommand{\bZ}{\mathbb{Z}}

\newcommand\subsetsim{\mathrel{%
\ooalign{\raise0.2ex\hbox{$\subset$}\cr\hidewidth\raise-0.8ex\hbox{\scalebox{0.9}{$\sim$}}\hidewidth\cr}}}

\DeclareMathOperator{\pr}{pr}

\DeclareMathOperator{\size}{size}

\DeclareMathOperator{\rank}{rank}

\DeclareMathOperator{\tors}{tors}

\newcommand{\signum}{\operatorname{sign}}
\newcommand{\eval}{\operatorname{ev}}
\newcommand{\str}{\operatorname{str}}

\newcommand{\norm}[1]{{\left\lVert #1\right\rVert}_\bZ}
\newcommand{\ornorm}[1]{{\left\lVert #1\right\rVert}^\mathrm{or}_\bZ}
\newcommand{\normo}[1]{\left\lVert #1\right\rVert}

\newtheorem{theorem}{Theorem}[section]
\newtheorem{corollary}[theorem]{Corollary}

\newtheorem{lemma}[theorem]{Lemma}

\theoremstyle{definition}
\newtheorem{definition}[theorem]{Definition}

\newtheorem{remark}[theorem]{Remark}
\newtheorem{example}[theorem]{Example}

\patchcmd{\subsection}{-.5em}{.5em}{}{}
\patchcmd{\subsubsection}{-.5em}{.5em}{}{}

\title[Counting maximally broken trajectories]{Counting maximally broken Morse trajectories on aspherical manifolds}
\author{Caterina Campagnolo}
\email{caterina.campagnolo@kit.edu}
\address{Institute for Algebra and Geometry, Karlsruhe Institute of Technology}

\author{Roman Sauer}
\email{roman.sauer@kit.edu}
\address{Institute for Algebra and Geometry, Karlsruhe Institute of Technology}
\thanks{The authors acknowledge funding by the Deutsche Forschungsgemeinschaft (DFG, German Research Foundation) – 281869850 (RTG 2229).}
\subjclass[2010]{Primary 57R99; Secondary 55N10}

\keywords{Morse-Smale function, torsion homology, broken Morse trajectories, simplicial norm}

\begin{document}
\begin{abstract}
We prove a lower bound on the number of maximally broken trajectories of the negative gradient flow of a Morse-Smale function on a closed aspherical manifold in terms of integral (torsion) homology. 
\end{abstract}

\maketitle

\section{Introduction}

A fundamental relationship between a Morse function $f\colon M\to \bR$ on a 
closed smooth manifold $M$ and its homology are the Morse inequalities. They relate 
the number $\nu_p(f)$ of critical points of index $p$ to the Betti numbers $b_p(M)$ of $M$ by the inequalities 
\[ \sum_{p=0}^n(-1)^{p+n}\nu_p(f)\ge \sum_{p=0}^n(-1)^{p+n}b_p(M)\]
for every $n\le \dim(M)$. 

The Morse inequalities follow by easy homological algebra on the cellular chain complex of a CW-complex $X$ that is homotopy equivalent to $M$ and has 
$\nu_p(f)$ many $p$-cells. Such a CW-complex always exists. If $f$ is a \emph{Morse-Smale function}, i.e.~a Morse function satisfying the Morse-Smale condition with respect to a Riemannian metric on $M$, then one can directly associate a chain complex with~$f$, the \emph{Morse-Smale-Witten complex} which computes the homology of~$M$. 

The goal of this paper is to describe the following relation between a Morse function~$f$ on a closed smooth aspherical manifold and its homology which does not seem obtainable from the homological algebra of the Morse-Smale-Witten complex alone (see Remark~\ref{rem: not from Witten complex}).

\begin{theorem}\label{thm: main theorem}
Let $M$ be a $d$-dimensional closed orientable aspherical Riemannian manifold. Let $f\colon M \to\bR$ be a Morse-Smale function. Then the number of maximally broken trajectories of $f$ is at least 
\[\frac{1}{2^{d+1}(d+1)(d+1)!}\cdot\sum_{p=0}^d \size\bigl( H_p(M;\bZ)\bigr).\]
\end{theorem}

Let us recall and define the notions appearing in the statement. 

If $A$ is an abelian group we define its \emph{size} $\size(A)\in [0,\infty]$ by 

\[\size(A)=\rank(A)+\log |\tors A|,\]
the sum of its rank and the logarithm of the cardinality of its torsion subgroup. 

An \emph{$n$-part broken trajectory} of $f$ is a sequence of critical points $x_n, \dots, x_0$ where $i$ is the index of $x_i$, and a sequence of (unparametrized) flow lines $\gamma_n,\dots,\gamma_1$ of $-\nabla f$ where each $\gamma_i$ runs from $x_i$ to $x_{i-1}$. The number of (unparametrized) flow lines between critical points of index $i$ and $i-1$ is always finite. See~\cite{morse-book} as a background reference for the relevant Morse theory. If $M$ is $d$-dimensional, then a $d$-part broken trajectory is also called a \emph{maximally broken trajectory}. 

The theorem above is inspired by two theorems. 
The first one by H.~Alpert involves the simplicial volume $\normo{M}$ of a closed manifold $M$. 
The simplicial volume is the infimum over the $\ell^1$-norms of real cycles representing the fundamental class. 

\begin{theorem}[Alpert~\cite{alpert}*{Theorem~18}]\label{thm: alpert theorem}
Let $M$ be a closed oriented aspherical Riemannian manifold. Let $f\colon M \to\bR$ be a Morse-Smale function. Then the number of maximally broken trajectories of $f$ is at least $ \normo{M}$, where $\normo{M}$ denotes the simplicial volume of $M$. 
\end{theorem}

With Gromov's and Thurston's computation of simplicial volume of hyperbolic manifolds~\cite{gromov} Alpert concludes: 

\begin{corollary}[\cite{alpert}*{Theorem~1}]\label{cor: Alpert hyperbolic}
	Let $M$ be a closed oriented hyperbolic $d$-manifold. Then the number of maximally broken trajectories of a Morse-Smale function on $M$ is at least $\operatorname{vol}(M)/v_d$, where $v_d$ is the volume of a regular ideal $d$-simplex in hyperbolic $d$-space. 
\end{corollary}

The second theorem is a statement about the integral simplicial volume $\norm{M}$ of a 
closed manifold $M$. The integral simplicial volume is the infimum over the $\ell^1$-norms of integral cycles representing the fundamental class. 

That $\norm{M}$ bounds the Betti numbers in all degrees was first observed by 
Gromov (see~\cite{luck}*{Example~14.28} for a first detailed proof). In~\cite{frigerioetal}*{Lemma~4.1} the improved rank bound below was proved. That $\norm{M}$ also bounds the torsion homology as below was proved by the second author in~\cite{sauer}*{Theorem~3.2}.

\begin{theorem}\label{thm: homology bounds}
Let $M$ be a closed oriented $d$-dimensional manifold. Then the following estimates hold for every $p\in\{0,\dots, d\}$:
\begin{enumerate}
\item $\rank H_p(M;\bZ)\le \norm{M}$,
\item $\log|\tors(H_p(M;\bZ))|\le \log(d+1)\binom{d+1}{p+1}\norm{M}$. 
\end{enumerate}
\end{theorem}

We follow Alpert's methods closely. In fact, this paper grew out of a seminar in which we read Alpert's paper and realized that her proof works integrally and can be combined with Theorem~\ref{thm: homology bounds} to show the lower (torsion) homology bound for the number of maximally broken trajectories.

In the next section we discuss some examples where  Theorem~\ref{thm: main theorem} but not Theorem~\ref{thm: alpert theorem} can be applied. Section \ref{norms on integral chains} is devoted to the comparison between the norms on the singular chain complex and the so-called oriented chain complex of a topological space. Section \ref{I orl} introduces the key lemma of the method in the version adapted to our case, namely the integral reduction lemma (Lemma \ref{General ARL}). Section \ref{norms morse-smale} deals with properties of norms related to cellular decompositions of manifolds coming from Morse-Smale functions. The proof of Theorem~\ref{thm: main theorem} is concluded in Section~\ref{sec: conclusion of proof}.

We would like to thank our former colleague Federico Franceschini with whom we started working on Alpert's methods. We also thank the anonymous referee for her or his useful remarks.

\section{Examples}

We discuss three geometric examples and applications. Each example exhibits a sequence $(M_n)_n$ of closed aspherical $3$-manifolds such that for any choice of Morse-Smale function $f_n$ on $M_n$ the number of maximally broken trajectories of $f_n$ tends to infinity. The first example is taken from Alpert's paper and can be concluded from Corollary~\ref{cor: Alpert hyperbolic}. The other two examples follow from Theorem~\ref{thm: main theorem}. \begin{example}[\cite{alpert}*{Proposition 2}]
Let $\phi\colon \Sigma\to\Sigma$ be a pseudo-Anosov diffeomorphism of a closed surface $\Sigma$ of genus at least two. 
Then the mapping tori $M_n=T(\phi^n)$ are closed hyperbolic $3$-manifolds that admit a Morse-Smale function 
$f_n\colon T(\phi^n)\to\bR$ such that the number of critical points of $f_n$ is uniformly bounded;
but $\operatorname{vol}(T(\phi^n))\to\infty$ and hence the number of maximally broken trajectories of $f_n$ tends to $\infty$. 
\end{example}

\begin{example}
Let $K$ be the figure eight knot in $S^3$. The knot complement $M=S^3-K$ is a complete hyperbolic $3$-manifold of finite volume. By various Dehn fillings on the torus boundary of $M$ one obtains a sequence of closed hyperbolic $3$-manifolds $M_n$ with $|\tors(H_1(M_n;\bZ))|\ge n$ and uniformly bounded volume (see~\cite{BGS}*{Theorem~1.7}). In fact, by Thurston's hyperbolic Dehn filling theorem the volume of $M_n$ is at most the volume of $M$. By Theorem~\ref{thm: main theorem} the number of maximally broken trajectories of $f_n$ for any sequence of Morse-Smale functions $f_n\colon M_n\to\bR$ grows at least like $\log(n)$. 
\end{example}

\begin{example}
For every $n\in\bN$ consider the group 
\[\Gamma_n=\langle x,y,z\mid [x,z]=[y,z]=1,~[x,y]=z^n\rangle.\]
One verifies that $\Gamma_n$ is a torsion-free nilpotent group: 
This can be seen from the fact that $\Gamma_n$ is just the variant of the integral $3$-dimensional Heisenberg group where the $(1,2)$-matrix entry is in $n\bZ$. 
 Let $M_n=\bR^3/\Gamma_n$ be the associated nil-manifold. In particular, each $M_n$ is a $3$-dimensional closed aspherical smooth $3$-manifold. 
The first homology group of $M_n$ equals the abelianization 
\[ \Gamma_n/[\Gamma_n,\Gamma_n]\cong \bZ^2\oplus \bZ/n.\]
In particular, the Betti numbers of each $M_n$ are $b_0=1, b_1=2, b_2=2, b_3=1$ by Poincar\'{e} duality and thus uniformly bounded. 
Fix a Riemannian metric on each $M_n$ and choose Morse-Smale functions $f_n\colon M_n\to\bR$ with respect to these metrics. The number of maximally broken trajectories of $f_n$ grows at least like 
$\log(n)$ by Theorem~\ref{thm: main theorem}. Note also that the simplicial volume of each $M_n$ vanishes since its fundamental group $\Gamma_n$ is amenable~\cite{gromov}. 
\end{example}

Finally, we would like to discuss why Theorem~\ref{thm: main theorem} is not an obvious  consequence of the Morse-Smale-Witten complex alone, as was pointed out before the statement of 
Theorem~\ref{thm: main theorem}. 

\begin{remark}\label{rem: not from Witten complex}
Since the Morse-Smale-Witten complex $W_\ast(M)$ associated to a Morse-Smale function $f\colon M\to \bR$ computes the singular homology of $M$ it is tempting to extract the 
bound in Theorem~\ref{thm: main theorem} from it by homological algebra. Each chain group $W_i(M)$ is a free abelian group whose rank is the number of critical points of index $i$. Let $p$ and $q$ be critical points of index $i$ resp.~$i-1$. The $(p,q)$-matrix entry of the differential $W_{i}(M)\to W_{i-1}(M)$ is the signed number of flow lines from $p$ to $q$. If we could bound all these numbers we would obtain bounds on the norm of the differentials and hence bounds on the (torsion) homology (cf.~\cite{sauer}*{Lemma~3.1}). But it could be a priori possible that the number of maximally broken trajectories is small but there are still two critical points $p,q$ as above with a large number of flow lines from $p$ to $q$. This prevents us from deducing Theorem~\ref{thm: main theorem} directly from the Morse-Smale-Witten complex. 

Other consequences of Theorem~\ref{thm: main theorem} are more elementary and can be deduced more directly. It follows from Theorem~\ref{thm: main theorem}, for example, that 
on a closed aspherical smooth manifold any Morse-Smale function has a critical point of index $p$ for every $p\in\{0,\dots, \dim(M)\}$. This, however, can be shown directly by basic homotopy theory: 

It suffices to show that any CW-structure on a closed aspherical manifold $M$ has cells in every degree $\le \dim M$. 
Let $X$ be a CW-structure on $M$. Suppose for $i<\dim X=\dim M$ we have $X^{(i)}=X^{(i+1)}$, i.e.~there are no $(i+1)$-cells. We can suppose that $i>0$; otherwise 
we had trivial $\pi_1(M)$ and $M$ were of dimension $0$. We can also suppose that $i>1$; if $i=1$ then $\pi_1(M)$ would be free. If it is free abelian, $M=S^1$ and the claim is verified; if not we have a contradiction, as there is no closed aspherical manifold with free nonabelian fundamental group since the homology of such a group does not satisfy Poincar\'e duality. Hence $i\ge 2$ and the inclusion $X^{(i)}\to X$ is a $\pi_1$-isomorphism. Let $p\colon \tilde X\to X$ be the universal covering. Note that $\tilde X^{(i)}=p^{-1}(X^{(i)})$ is the universal covering of $X^{(i)}$. Since $\tilde X^{(i)}=\tilde X^{(i+1)}\to \tilde X $ is $(i+1)$-connected the homology satisfies 
\[ H_\ell\bigl(\tilde X^{(i)}\bigr)=0~\text{ for all $0<\ell\le i$.}\] 
On the other hand, $H_\ell(\tilde X^{(i)})=0$ for $\ell>i$ since $\tilde X^{(i)}$ is $i$-dimensional. By Whitehead's theorem $\tilde X^{(i)}$ is contractible and  thus $X^{(i)}$ is aspherical with fundamental group $\pi_1(X)=\pi_1(M)$. Hence the inclusion $X^{(i)}\to X$ is a homotopy equivalence. Since the $\bZ/2$-homology of $X$ does not vanish in degree $\dim M$ and $i\leq \dim M$, this implies that $i=\dim M$. 	
\end{remark}

\section{Norms on integral chains}\label{norms on integral chains}

The integral singular chain complex of a topological space $X$ is denoted by $C_\ast(X)$. The \emph{integral norm} $\norm{c}$ of a chain $c=\sum_{\sigma}a_\sigma\sigma\in C_p(X)$, where $\sigma$ runs over the singular $p$-simplices of $X$, is defined as 
\[ \norm{c}=\sum_{\sigma} |a_\sigma|.\]
The symmetric group $S(p+1)$ in $p+1$ letters acts on the set $\cS_p(X)$ of singular $p$-simplices of $X$ through the affine-linear extension of the permutation action on the vertices. Let $O_p(X)\subset C_p(X)$ be the submodule generated by the union of  
\[ \bigl\{ g\sigma-\signum(g)\sigma\mid \sigma\in \cS_p(X), g\in S(p+1)\bigr\}\]
and 
\[   \bigl\{ \sigma\in\cS_p(X)\mid \text{there is a transposition $t$ with $t\sigma=\sigma$}\bigr\}.
\]
It is straightforward to verify that $O_\ast(X)$ is a subcomplex~\cite{barr-paper}*{Proposition~2.2}. We denote by $C_\ast^o(X)$ the quotient complex $C_\ast(X)/O_\ast(X)$. Obviously, $C_\ast^o$ provides a functor from topological spaces to chain complexes of abelian groups. The natural projection $\pr_\ast\colon C_\ast\to C_\ast^o$ is a natural chain homomorphism. 
For an element $c\in C_p^o(X)$ we let the \emph{integral oriented norm} 
$\ornorm{c}$ be defined as 
\[ \ornorm{c}=\min\bigl\{ \norm{c'}\mid \pr_p(c')=c\bigr\}.\]

Further, both the integral and the integral oriented norm induce functions 
on $H_p(C_\ast(X))=H_p(X)$ and $H_p(C_\ast^o(X))$, respectively, by taking the minimum among representing cycles; we also call these functions \emph{integral norm} and \emph{integral oriented norm}, respectively, and denote them by the same symbols as their analogues on the chain complexes. 
\begin{definition}\label{def: integral simplicial volume}
The \emph{integral simplicial volume} of a closed oriented manifold $M$ is the integral norm of its fundamental class. It is denoted by $\norm{M}$. 
\end{definition}

The oriented singular chain complex $C_\ast^o$ historically predates the singular chain complex $C_\ast$ and was introduced in 1944 by Eilenberg. A comparison between the two should have been discussed a long time ago but was only addressed in writing in a 1995 paper by Barr where the following 
result is proved~\cite{barr-paper}*{Theorem~1.1}\footnote{As the MathSciNet reviewer of~\cite{barr-paper} puts it: \emph{This paper discusses a strange historical lacuna}.}.  

\begin{theorem}\label{thm: comparision non-oriented and oriented homology}
The projection $\pr_\ast$ induces an isomorphism in homology. 
\end{theorem}

To obtain the next theorem we basically have to reprove Theorem~\ref{thm: comparision non-oriented and oriented homology} with greater care.  

\begin{theorem}\label{thm: comparision of norms}
Let $c$ be a singular $p$-cycle and $[c]$ its homology class. Then 
\[ \ornorm{[\pr_p(c)]}\le \norm{[c]}\le (p+1)!\cdot\ornorm{[\pr_p(c)]}.\]
\end{theorem}

\begin{remark}
The proof constructs inductively a (natural) chain homotopy inverse $\phi\colon C_\ast^o(X)\to C_\ast(X)$ of the projection $\pr_\ast\colon C_\ast(X)\to C_\ast^o(X)$. Geometrically, the map $\phi$ takes a singular $p$-simplex $\sigma$ (defined up to permutation) and maps it to a linear combination of $(p+1)!$ singular $p$-simplices coming from the barycentric subdivision of $\sigma$. Each simplex in the barycentric subdivision is associated to a chain of faces of $\sigma$ and thus has a well defined order according to the dimension of the face each of its vertices came from. The coefficients of the linear combination are in $\{1,-1\}$; the sign depends on the order in which vertices are added to or excluded from the faces. To work out the details we benefit from a more abstract approach; see the proof below. For an explicit low-degree example of what $\phi$ does see~\cite{barr-paper}*{Section~5}.\end{remark}

\begin{proof}[Proofs of Theorems~\ref{thm: comparision non-oriented and oriented homology} and~\ref{thm: comparision of norms}]

For a topological space $X$ and $x\in X$ let $P_x(X)$ be the space of continuous paths $[0, 1]\rightarrow X$ starting from $x$ endowed with the compact-open topology. Let $G(X)$ be the topological sum of $P_x(X)$ running over all points $x$ in $X$. In a natural way, $X\mapsto G(X)$ becomes an endofunctor of the category of topological spaces. Evaluation at $1$ yields a natural transformation $\eval$ from $G$ to the identity functor. 

For any space $Y$ let $\epsilon\colon C_0(Y)\to\bZ$ denote the augmentation map sending each singular $0$-simplex to $1$. Note that $C_0^o(Y)=C_0(Y)$. We set $C_{-1}(Y)=C_{-1}^o(Y)=\bZ$. 

To make sense of the next statement we will think of $S(p+1)$ being embedded into $S(p+2)$ and consisting of those permutations of $\{0,\dots,p+1\}$ that fix $p+1$. According to~\cite{barr-paper}*{Proposition~4.1} there is a natural and $S(p+1)$-equivariant chain contraction 
\[ s_p\colon C_p(G(X))\to C_{p+1}(G(X)), ~p\ge -1, \]
of the (augmented) chain complex 
\[ \dots\to C_2(G(X))\to C_1(G(X))\to C_0(G(X))\xrightarrow{\epsilon}\bZ\to 0.\]
By connectedness, a singular $p$-simplex $\sigma\in\cS(G(X))$ lies in exactly one component $P_x(X)$ of $G(X)$. We say  
that $x$ is the base point of $\sigma$ and denote it by $b_\sigma=x$. We may view $\sigma$ as a continuous map $\Delta^p\times [0,1]\to X$. We use barycentric coordinates for the points of a standard simplex. The explicit formula for $s$ applied to a singular $p$-simplex in~\cite{barr-paper} is 
\[ s_p(\sigma)(t_0,\dots, t_{p+1}; u)=\begin{cases} 
       \sigma(\frac{t_0}{1-t_{p+1}}, \dots, \frac{t_p}{1-t_{p+1}}; (1-t_{p+1})u) & \text{ if $t_{p+1}\ne 1$,}\\ b_\sigma&\text{ otherwise,}\end{cases}
\]
provided $p>0$. There is a similar expression for $p=0$. Therefore we see that 
$s_p(\sigma)$ is a singular $(p+1)$-simplex for every singular $p$-simplex $\sigma$. 
By naturality we obtain that 
\begin{equation}\label{eq: 1-Lipschitzness of chain contraction}
\norm{s_p(c)}\le \norm{c}
\end{equation}
for every chain $c\in C_p(G(X))$. 

Further, for every $p\ge 0$ there is a natural and $S(p+1)$-equivariant homomorphism $\theta_p\colon C_p (X)\to C_p(G(X))$ 
such that $\theta_p$ followed by $C_p(\eval)$ is the identity~\cite{barr-paper}*{Proposition~4.2}. 
Note that we do \emph{not} claim that $\theta_\ast$ is a chain map. 
The explicit formula for $\theta_p(\sigma)$ where $\sigma\in\cS_p(X)$ is 
\[ 
\theta_p(\sigma)(t_0,\dots, t_p;u)=\sigma\Bigl( ut_0+\frac{1-u}{p+1},\dots,ut_p+\frac{1-u}{p+1}\Bigr).
\]
So 
the basepoint of $\theta_p(\sigma)$ is the barycenter of $\sigma$. 
We see that $\theta_p$ maps every singular $p$-simplex of $X$ to a singular $p$-simplex of $G(X)$. As above, this implies an estimate 
\begin{equation}\label{eq: 1-Lipschitzness of theta}
\norm{\theta_p(c)}\le \norm{c}
\end{equation}
for every chain $c\in C_p(X)$. For every $p\ge 0$ we have 
\[ s_p\bigl (O_p(G(X))\bigr)\subset O_{p+1}(G(X))\]
as an immediate consequence of equivariance. 
Hence $s_\ast$ descends to a natural chain contraction 
$s_\ast^o$ of 
\[ \dots\to C_2^o(G(X))\to C_1^o(G(X))\to C_0^o(G(X))\xrightarrow{\epsilon}\bZ\to 0.\]
Moreover, each $\theta_p$ descends to a natural homomorphism \[\theta_p^o\colon C_p^o(X)\to C_p^o(G(X))\] that is a right inverse of $C_p^o(\eval)$.  
The $1$-Lipschitz property in~\eqref{eq: 1-Lipschitzness of chain contraction} or~\eqref{eq: 1-Lipschitzness of theta} holds true for $s_p^o$ or $\theta_p^o$ as well. 

Next we construct a natural chain homomorphism $\phi\colon C_\ast^o(X)\to C_\ast(X)$ that is a chain homotopy inverse to the projection $\pr_\ast$ such that 
\[ \norm{\phi_p(c)}\le (p+1)!\cdot\ornorm{c}.\] 
This will finish the proof. Suppose that we have for any space $X$ and every $0\le i\le p$ 
natural chain homomorphisms $\phi_i\colon C_i(X)^o\to C_i(X)$ such that 
\begin{equation}\label{eq: ladder}
\begin{tikzcd}
 C_p^o(X)\arrow[r,"\partial^o_X"]\arrow[d, "\phi_p"] &  C_{p-1}^o(X)\arrow[r,"\partial^o_X"]\arrow[d, "\phi_{p-1}"] & \dots\arrow[r,"\partial^o_X"] & C_1^o(X)\arrow[r,"\partial^o_X"]\arrow[d, "\phi_1"] & C_0^o(X)\arrow[r]\arrow[d,"\phi_0"] & \bZ\arrow[d, "\mathrm{id}"]\\
 C_p(X)\arrow[r, "\partial_X"] & C_{p-1}(X)\arrow[r, "\partial_X"] &\dots\arrow[r, "\partial_X"] & C_1(X)\arrow[r, "\partial_X"] & C_0(X)\arrow[r] & \bZ
\end{tikzcd}
\end{equation}
commutes and such that $\phi_i$ is $(i+1)!$-Lipschitz for every $i\in\{0,\dots,p\}$. We define $\phi_{p+1}$ as the 
composition 
\begin{multline}\label{eq: definition of phi} C_{p+1}^o(X)\xrightarrow{\theta_{p+1}^o} C_{p+1}^o(G(X))\xrightarrow{\partial^o_{G(X)}} C_p^o(G(X))\xrightarrow{\phi_p} C_p(G(X))\\\xrightarrow{s_p} C_{p+1}(G(X))\xrightarrow{C_{p+1}(\eval)} C_{p+1}(X).
\end{multline}	
One verifies that $\phi_{p+1}$ adds another commutative square to~\eqref{eq: ladder} by 
the following computation: 
\begin{align*}
	\partial_X\circ\phi_{p+1} &= \partial_X\circ C_{p+1}(\eval)\circ s_p\circ \phi_p\circ\partial^o_{G(X)}\circ\theta^o_{p+1}\\
	&=C_p(\eval)\circ\partial_{G(X)}\circ s_p\circ\phi_p\circ \partial^o_{G(X)}\circ\theta_{p+1}^o\\
	&=C_p(\eval)\circ\bigl( \operatorname{id}-s_{p-1}\circ\partial_{G(X)}\bigr)\circ \phi_p\circ \partial^o_{G(X)}\circ \theta^o_{p+1}\\
	&=C_p(\eval)\circ\phi_p\circ\partial^o_{G(X)}\circ\theta^o_{p+1}-\underbrace{C_p(\eval)\circ s_{p-1}\circ \partial_{G(X)}\circ\phi_p\circ\partial^o_{G(X)}\circ\theta_{p+1}^o}_{=0\text{ since $\partial_{G(X)}\circ\phi_p=\phi_{p-1}\circ\partial^o_{G(X)}$}}	\\
	&=C_p(\eval)\circ\phi_p\circ\partial^o_{G(X)}\circ\theta^o_{p+1}\\
	&=\phi_p\circ C_p^o(\eval)\circ \partial^o_{G(X)}\circ\theta^o_{p+1}\\
	&=\phi_p\circ \partial_X^o\circ C_{p+1}^o(\eval)\circ\theta^o_{p+1}\\
	&=\phi_p\circ \partial_X^o.
\end{align*}

Further, since the boundary homomorphism $C_{p+1}^o(G(X))\to C_p^o(G(X))$ is $(p+2)$-Lipschitz and $\phi_p$ is $(p+1)!$-Lipschitz and the other homomorphisms appearing in~\eqref{eq: definition of phi} are $1$-Lipschitz the map $\phi_{p+1}$ is $(p+2)!$-Lipschitz. 

Finally, to show that there are natural chain homotopies for $\pr_\ast\circ \phi_\ast$ and $\phi_\ast\circ\pr_\ast$ is similar to constructing $\phi_\ast$. It is also a consequence 
of~\cite{barr-paper}*{Theorem~1.1}. 
\end{proof}
\section{Integral reduction lemma}\label{I orl}
In this section we prove the integral oriented version of the amenable reduction lemma \cite{alpert}*{Lemma 4, Corollary 5}, \cite{alpert+katz}*{Lemma~4}. It will allow us to estimate the integral oriented norm of cycles in terms of some conditions on the shape of the simplices composing them.

By~\cite{alpert+katz}*{Lemma~2} (which attributes it to~\cite{gromov}*{p.~48}) the singular chain complex of an aspherical space $Z$ admits a straightening operator 
\[\str_\ast: C_\ast(Z)\rightarrow C_\ast(Z),\]
i.e.~a chain homomorphism $\str_\ast$ such that 
\begin{enumerate}
	\item $\str_\ast$ is chain homotopic to the identity, 
	\item $\str_p$ is equivariant with regard to the action of the symmetric group $S(p+1)$ on $C_p(Z)$, and  
	\item if the singular $p$-simplices $\sigma$ and $\sigma'$ have the same sequence of vertices and their lifts to the universal cover have the same sequence of vertices then 
	$\str_p(\sigma)=\str_p(\sigma')$. 
\end{enumerate}

By the second property the straightening operator descends to a well-defined chain homomorphism 
\[\mathrm{str}_\ast^o: C_\ast^o(Z)\rightarrow C_\ast^o(Z).\]

\begin{remark}\label{rem:str o * -id}
For any cycle $c\in Z_p^o(Z)$, we have $[\mathrm{str}_p^o(c)]=[c]\in H_p(Z)$. Indeed, by Theorem~\ref{thm: comparision non-oriented and oriented homology} the map $\pr_\ast$ induces an isomorphism on homology. Therefore there exists $\tilde{c}\in Z_p(Z)$ such that $\pr_p([\tilde{c}])=[c]$. Denote by $h_\ast\colon C_\ast(Z)\rightarrow C_{\ast+1}(Z)$ the homotopy between $\mathrm{str}_\ast$ and the identity on $C_\ast(Z)$. Then
\begin{align*}
[\mathrm{str}_p^o(c)]-[c]=[\mathrm{str}_p^o\pr_p(\tilde{c})]-[\pr_p(\tilde{c})]&=[\pr_p\mathrm{str}_p(\tilde{c})]-[\pr_p(\tilde{c})]\\		&=[\pr_p(\mathrm{str}_p(\tilde{c})-\tilde{c})]\\
&=[\pr_p(\partial h_p (\tilde{c})+h_{p-1}\partial(\tilde{c}))]\\
	&=[\partial^o(\pr_p(h_p(\tilde{c}))]=0.
\end{align*}
\end{remark}
\begin{definition}\label{def: simplicial triple}
Let $X$ be a topological space. We call a triple $(\Sigma, c_{\Sigma}, \rho)$ consisting of a $p$-dimensional $\Delta$-complex $\Sigma$ which is a pseudomanifold, a simplicial fundamental cycle $c_\Sigma$ on $\Sigma$ and a continuous map $\rho\colon \Sigma\to X$ a \emph{simplicial triple} of $X$. 
For a singular $p$-chain $c$ on $X$ we say that a simplicial triple $(\Sigma, c_{\Sigma}, \rho)$ is a \emph{simplicial $c$-triple} or a \emph{simplicial triple representing $c$} if $\rho_\ast(c_\Sigma)=c$. 
\end{definition}

It is well known that singular cycles can always be represented by 
simplicial triples~\citelist{\cite{alpert}*{Section~2}\cite{hatcher}*{p.~108}}.

\begin{definition}
A \emph{partial coloring} of a $\Delta$-complex $\Sigma$ is a coloring of a subset of its vertices. The \emph{induced partial coloring} on edges of the $\Delta$-complex colors an edge with color $\ell$ if and only if both of its endpoints are colored with $\ell$. 
A simplicial triple is called \emph{partially colored} if its underlying $\Delta$-complex is endowed with a partial coloring. 
Given a partial coloring of the simplicial triple $(\Sigma, c_\Sigma, \rho)$ of $X$, we say that a simplex $\Delta$ of $\Sigma$ is \emph{non-essential} if either
\begin{enumerate}
\item $\Delta$ has two distinct vertices with the same color, or 
\item $\Delta$ has two vertices with the same $\rho$-image and the $\rho$-image of their  connecting edge is a homotopically trivial loop in~$X$. 
\end{enumerate}
If neither of these hold, $\Delta$ is called \emph{essential}. Finally, we call a singular simplex appearing 
in $\rho_\ast(c_\Sigma)$ \emph{essential} if it is the $\rho$-image of an essential simplex. 
\end{definition}

The following statement is the integral analogue of \cite{alpert}*{Lemma 4, Corollary 5} and of ~\cite{alpert+katz}*{Lemma~4}.

\begin{lemma}[Integral reduction lemma]\label{General ARL}
Let $\tilde{c}=\sum_\sigma a_\sigma\sigma\in C_p(X)$ be a $p$-cycle represented by a partially colored simplicial triple $(\Sigma, \tilde{c}_\Sigma, \rho)$. Let $\alpha\colon  X\rightarrow Z$ be a continuous map to an aspherical space $Z$ such that the  restriction of $\alpha\circ\rho\colon \Sigma\to Z$ to every component of every $1$-dimensional subcomplex given by edges of the same color induces the trivial map on fundamental groups. Let $c=\pr_p(\tilde c)\in C_p^o(X)$. 
Then we have 
\begin{equation}\label{eq: essential simplices}
\ornorm{\alpha_\ast[c]}\leq\sum_\text{$\sigma$ essential}|a_\sigma|.
\end{equation}
\end{lemma}
\begin{proof}
In the first part of the proof we reduce the statement to the special case 
where $\rho$ maps all vertices of $\Sigma$ to the same point $x_0\in X$. 

By introducing new colors for the connected components we may assume 
that each $1$-dimensional subcomplex of edges of the same color is connected. This does not change the subset of essential singular simplices.

We will homotope the map $\rho\colon \Sigma\rightarrow X$ to a map $\rho'\colon\Sigma\rightarrow X$ such that $\rho'$ maps all vertices of $\Sigma$ to the same point of $X$ and $\alpha\circ\rho'$ maps every colored edge to a nullhomotopic loop in $Z$. 

To this end, we pick for every color $\ell$ a spanning tree $T(\ell)$ of the $1$-dimensional subcomplex $\Sigma(\ell)\subset\Sigma$ of $\ell$-colored edges in $\Sigma$. Choose a vertex $v(\ell)$ in $T(\ell)$ for every color~$\ell$. We pick paths from every $\rho(v(\ell))$ and from the $\rho$-image of every uncolored vertex to $x_0$. These paths define a homotopy on the subset of these vertices to $X$ which we extend to a homotopy $G\colon\Sigma\times~[0,1]\to X$ using the fact that a subcomplex in a simplicial complex is a cofibration.

 Since $T(\ell)\subset \Sigma(\ell)$ is a contractible subset and an inclusion which is a cofibration, 
the projection $\pr(\ell)\colon\Sigma(\ell)\to \Sigma(\ell)/T(\ell)$ is a homotopy equivalence and thus has a homotopy inverse $j(\ell)$ which maps the basepoint given by $T(\ell)$ to $v(\ell)$. Again by the cofibration property we can extend the homotopy 
\[\Bigl(\coprod_\ell \Sigma(\ell)\coprod\{\text{uncolored vertices}\}\Bigr)\times [0,1]\to \Sigma(\ell)\hookrightarrow\Sigma\xrightarrow{\rho} X\] 
which is constant $\rho$ on the uncolored vertices and is a homotopy between $\rho$ and $\coprod \rho\circ j(\ell)\circ \pr(\ell)$ on $\Sigma(\ell)$ to a homotopy $H\colon \Sigma\times [0,1]\to X$ with 
$H_0=\rho$. Then the concatenation of the homotopies $H$ and $G$ is a homotopy from $\rho$ to a map $\rho'$ that maps all vertices to $x_0$. It is immediate from the construction that the restriction of $\rho'$ to $\Sigma(\ell)$ is still trivial on fundamental groups. The simplicial triple $(\Sigma, \tilde{c}_\Sigma, \rho')$ represents $c'=\rho'_\ast(\tilde{c}_\Sigma)$. The left and right hand sides of~\eqref{eq: essential simplices} are unchanged if we replace $c$ by $c'$. So we might as well assume now the original map $\rho$ maps all vertices of $\Sigma$ to the same point $x_0\in X$.

Consider the straightened cycle $\str_p(\alpha_\ast(\tilde{c}))\in C_p(Z)$. It is homologous to $\alpha_\ast\tilde{c}$ in $C_\ast(Z)$; by Remark \ref{rem:str o * -id} $\mathrm{str}_p^o(\alpha_\ast c)$ is homologous to $\alpha_\ast c$ in $C_\ast^o(Z)$. So we will rather estimate $\ornorm{\mathrm{str}^o_p(\alpha_\ast(c))}$. 

The result follows once we prove that for every non-essential singular simplex $\sigma$ in the linear combination $\tilde c=\sum_\sigma a_\sigma\sigma$ we have 
\begin{equation}\label{eq: non-essential vanishing}
\str_p^o\bigl(\alpha_\ast(\pr_p(\sigma))\bigr)=\pr_p\bigl(\str_p(\alpha_\ast\sigma)\bigr)=0. 
\end{equation}

Let $\sigma_{0}$ be a non-essential simplex of $\tilde{c}$. To show~\eqref{eq: non-essential vanishing} we distinguish two cases. 
\begin{enumerate}
\item $\sigma_0$ maps the distinct vertices $m,n\in\{0,\dots, p\}$ to the same point in $X$ 
with the loop of the corresponding edge being null-homotopic in~$X$, thus in $Z$:\\ Let $(n,m)\in S(p+1)$ be the transposition. Then 
\[\str_p(\alpha_\ast\sigma_{0})=\str_p((n,m)\alpha_\ast\sigma_{0})=(n, m)\str_p(\alpha_\ast\sigma_{0})\]
by the second and third property of the straightening operator. Hence $\str_p(\alpha_\ast\sigma_0)\in O_p(Z)$ and~\eqref{eq: non-essential vanishing} is implied. 
\item $\sigma_{0}$ has two distinct vertices of the same color:\\ 
Then the loop in $Z$ resulting from the edge connecting these vertices is null-homotopic by assumption, and we conclude $\str_p(\alpha_\ast\sigma_0)\in O_p(Z)$ as before. \qedhere
\end{enumerate}

\end{proof}

\section{Norms from the cellular structure of a Morse-Smale function}\label{norms morse-smale}

On a CW-complex we consider the smallest partial order $\le$ on the set of open cells such that 
$e\cap \bar e'\ne\emptyset$ for open cells $e$ and $e'$ implies $e\le e'$. Two open cells are \emph{incomparable} if they are incomparable with respect to $\le$.

\begin{definition}
We define four conditions on a cycle $c\in C_\ast(X)$ in a $CW$-complex $X$. Let $(\Sigma, c_\Sigma, \rho)$ be a simplical $c$-triple. 
\begin{enumerate}
\item The \emph{cellular} condition requires that for each simplex of $\Sigma$, the image of the interior of each face (of any dimension) must be contained in one cell of $X$.
\item The \emph{order} condition requires that the image of each simplex of $\Sigma$ must be contained in a totally ordered chain of cells; that is, the simplex does not intersect any two incomparable cells.
\item The \emph{internality} condition requires that for each simplex of $\Sigma$, if the boundary of a face (of any dimension) maps into a cell $e\subset X$, then the whole face maps into $e$.
\item The \emph{loop} condition requires that if a simplex in $\Sigma$ maps two vertices to the same point in $X$ then it maps the corresponding edge to that point. 
\end{enumerate}
It is easy to see that these conditions do not depend on the choice of simplicial triple. 
\end{definition}
\begin{definition}\label{def: essential cells}
Let $X$ be a $CW$-complex. Let $(\Sigma, c_\Sigma, \rho)$ be a simplicial triple representing a cycle $c=\sum a_\sigma\sigma\in C_p(X)$. Each open cell in $X$ is assigned its own color, and we color a  vertex $v$ of $\Sigma$ with the color of the open cell containing $\rho(v)$. We define 
the \emph{essential cellular norm} of $c$ by 
\[\|c\|_ {\mathbb{Z}, \mathrm{ess}}^{\mathrm{CW}}=\sum_{\sigma \mbox{ essential}}|a_\sigma|.\]
The \emph{essential cellular norm} of a (relative) homology class $h$ of $X$, which we also denote 
by~$\|h\|_ {\mathbb{Z}, \mathrm{ess}}^{\mathrm{CW}}$, is defined by taking the infimum of the essential cellular norms of chains that represent $h$ and satisfy the cellular, order, internality and loop conditions. 
\end{definition}

\begin{lemma}\label{CW-complex localisation}
Let $X$ be a finite $CW$-complex. Let $Z$ be any aspherical topological space and let $\alpha: X\rightarrow Z$ be a continuous map. For any homology class $h\in H_p(X)$, let $h_\mathrm{rel}$ denote the corresponding relative homology class in $H_p(X, X^{p-1})$. Then
$$\ornorm{\alpha_\ast h}\leq\|h_\mathrm{rel}\|_{\mathbb{Z}, \mathrm{ess}}^{\mathrm{CW}}.$$
\end{lemma}
\begin{proof}
The statement is the analog of Lemma~3 in Alpert's paper~\cite{alpert} when one replaces the simplicial norm of $\alpha_\ast h$ by its integral oriented norm and considers the (generalized) stratification of $X$ by open cells. We follow Alpert's proof closely and just point out the necessary modifications. Let $A=X^{p-1}$. Let $c_\mathrm{rel}$ be a relative cycle representing $h_\mathrm{rel}$ that satisfies the cellular, order, internality and loop conditions. By adding an integral singular chain $c_A$ in $A$ we
obtain a cycle $c_\mathrm{rel}+c_A\in C_p(X)$. From the chains $c_\mathrm{rel}$ and $c_A$ Alpert 
produces chains $c_\delta, c_1,c_2\in C_p(X)$ with $c=c_\mathrm{rel}+c_\delta+c_1+c_2\in C_p(X)$ being a cycle and a partial coloring of a simplicial $c$-triple $(\Sigma, c_\Sigma, \rho)$ 
such that 
\begin{itemize}
\item no singular simplex appearing in $c_\delta, c_1$ or $c_2$ is essential, and 
\item every essential singular simplex in $c_\mathrm{rel}$ is also essential with regard to the coloring in Definition~\ref{def: essential cells}. 
\end{itemize}
Having thus produced a representing cycle for $h$ that satisfies the conditions required in Lemma~\ref{General ARL}, we apply this result.
\end{proof}

In the following let $M$ be a closed Riemannian $d$-manifold. Let $f\colon M\to \bR$ be a Morse-Smale function that is, in addition, Euclidean. The latter is a technical assumption; see~\cite{quin}*{Definition~2.16} and~\cite{alpert}*{Section~3} for more details on this notion. If $f\colon M\to\bR$ is Morse-Smale and Euclidean, then the set of descending manifolds $\{\cD(p)\mid \text{ $p\in M$ critical}\}$ is a 
CW-decomposition of~$M$~\citelist{\cite{alpert}*{Lemma~11}\cite{quin}*{Section~3.4}}. 

\begin{definition}
Let $f\colon M\to\bR$ be a Euclidean Morse-Smale function on a closed Riemannian $d$-manifold. 
The essential cellular norm of a homology class $h\in H_p(M)$ with regard to the CW-structure of descending manifolds is denoted by $\|h\|_f$. 
\end{definition}

The following lemma is the crucial point in which the number of maximally broken trajectories is related to the essential cellular norm of the descending manifolds $\mathcal{D}(p)$ of the Morse function. It is proved in detail in~\cite{alpert}.

\begin{lemma}[\cite{alpert}*{Lemma 11}]\label{count n-broken traj}
Let $M$ be a closed Riemannian $d$-manifold, and let $f\colon M\rightarrow \mathbb{R}$ be a Euclidean Morse-Smale function. For each descending manifold $\mathcal{D}(p)$ of dimension $d$, let $[\mathcal{D}(p)]\in H_d(M, M^{d-1})$ be the corresonding relative homology class. Then the number of maximally broken trajectories of $f$ starting at the critical point $p$ is at least $\|[\mathcal{D}(p)]\|_f$. 
\end{lemma}

\section{Proof of the main theorem}\label{sec: conclusion of proof}

\begin{theorem}\label{thm: alperts main thm integrally}
Let $M$ be a closed oriented Riemannian $d$-manifold. Let $f\colon M\rightarrow \mathbb{R}$ be a Morse-Smale function. Let $Z$ be any aspherical topological space, and let $\alpha\colon M\rightarrow Z$ be a continuous map. Then the number of maximally broken 
trajectories of $f$ is at least 
$\ornorm{\alpha_\ast[M]}$.
\end{theorem}

\begin{proof}
By Theorem~10 and its subsequent remark in~\cite{alpert} we may and will assume that $f$ is in addition Euclidean. 

 Consider the $CW$-structure on $M$ given by the descending manifolds of~$f$.
The relative class $[M]_\mathrm{rel}\in H_d(M, M^{d-1})$ corresponding to $[M]$ is equal to the sum of $d$-cells
\[[M]_\mathrm{rel}=\sum_{\mathrm{ind}(p)=d}[\mathcal{D}(p)].\]
So by triangle inequality and Lemma~\ref{count n-broken traj} we have
\begin{align*}
\|[M]_\mathrm{rel}\|_f&\leq\sum_{\mathrm{ind}(p)=d}\|[\mathcal{D}(p)]\|_f\\&
\leq|\left\{\text{maximally broken trajectories}\right\}|.
\end{align*}
Applying Lemma \ref{CW-complex localisation}, we obtain
\[\ornorm{\alpha_\ast[M]}\leq\|[M]_\mathrm{rel}\|_f.\qedhere\]
\end{proof}

\begin{proof}[Proof of Theorem~\ref{thm: main theorem}]
Let $M$ be a closed orientable aspherical Riemannian $d$-manifold with a Morse-Smale function $f\colon M\to\bR$. We apply Theorem~\ref{thm: alperts main thm integrally} with $Z=M$ and $\alpha$ being the identity map. Hence the number of maximally broken trajectories is at least $\ornorm{[M]}$, where $[M]$ denotes the fundamental class. 
By Theorem~\ref{thm: comparision of norms} we have 
\[ \norm{[M]}\le (d+1)!\cdot\ornorm{[M]}.\]
By Theorem~\ref{thm: homology bounds} we obtain that 
\begin{align*} \log|\tors(H_p(M;\bZ))|&\le \log(d+1)\binom{d+1}{p+1}\norm{[M]},\\
 \rank(H_p(M;\bZ))&\le\norm{[M]},
\end{align*}
for every $p\in\{0,\dots, d\}$. 
Now we can conclude the theorem from
\begin{align*}
\sum_{p=0}^d\size\bigl(H_p(M;\bZ)\bigr)&\le \sum_{p=0}^d \Bigl( \log(d+1)\binom{d+1}{p+1}+1\Bigr)\norm{[M]}\\
&=\Bigl(\log(d+1)\sum_{p=0}^d \binom{d+1}{p+1}+ (d+1)\Bigr)\norm{[M]}\\
&=\Bigl(\log(d+1)\bigl(\Bigl(\sum_{p=0}^{d+1} \binom{d+1}{p}\Bigr)-1\bigr)+ (d+1)\Bigr)\norm{[M]}\\
&\le\Bigl((d+1)\bigl(2^{d+1}-1\bigr)+ (d+1)\Bigr)\norm{[M]}\\
&=2^{d+1}(d+1)\norm{[M]}\\
&\le 2^{d+1}(d+1)(d+1)!\cdot\ornorm{[M]}.\qedhere
\end{align*}
\end{proof}

\end{document}